\documentclass[11pt, a4paper]{amsart}
\usepackage{amssymb,amsmath}
\usepackage{amsthm, dsfont, a4}
\usepackage{hyperref}
\usepackage[all]{xy}
\usepackage[active]{srcltx}
\usepackage[short,nodayofweek]{datetime}
\newcommand\comment[1]{}

\renewcommand\theenumi{(\roman{enumi})}

\renewcommand\labelenumi{\theenumi}

\def\NN{{\mathbb N}}

\def\QQ{{\mathbb Q}}
\newcommand{\RR}{{\mathbb R}}
\renewcommand{\AA}{{\mathbb A}}

\def\O{{\mathcal O}}

\def\id{{\rm id}}

\def\vect#1{\text{\boldmath $#1$\unboldmath}}

\renewcommand{\th}[1]{\theta^{(#1)}}
\newcommand{\uob}{\mathds{1}}
\newcommand{\partdef}[1]{ \left\{ \begin{array}{ll} #1 \end{array} \right. }

\newcommand{\Jac}[2]{J(\vect{#1},\vect{#2})}
\DeclareMathOperator{\ord}{ord}

\def\markdef{\bf }

\theoremstyle{plain}
\newtheorem{thm}{Theorem}[section]
\newtheorem{cor}[thm]{Corollary}
\newtheorem{lem}[thm]{Lemma}
\newtheorem{prop}[thm]{Proposition}

\theoremstyle{definition}
\newtheorem{defn}[thm]{Definition}
\newtheorem{exmp}[thm]{Example}
\newtheorem{rem}[thm]{Remark}

\newtheoremstyle{Acknowledgements}
  {}
    {}
     {}
     {}
    {\bfseries}
    {}
     {.5em}
    {\thmname{#1}\thmnumber{ }\thmnote{ (#3)}}
\theoremstyle{Acknowledgements}

\begin{document}

\title[chain rule formula and polynomial inverses]{A chain rule formula for higher derivations and inverses of polynomial maps}

\author{Andreas Maurischat}
\address{\rm {\bf Andreas Maurischat}, Lehrstuhl A f\"ur Mathematik, RWTH Aachen University, Germany }
\email{\sf andreas.maurischat@matha.rwth-aachen.de}
\thanks{}



\date{$7^{\rm th}$ Nov, 2016}

\begin{abstract}
The multidimensional chain rule formula for analytic functions  and its generalisation to higher derivatives perfectly work in the algebraic setting in characteristic zero. In positive characteristic one runs into problems due to denominators in these formulas.

In this article we show a direct analog of these formulas using higher derivations which are defined in any characteristic. We also use these formulas to show how higher derivations to different coordinate systems are related to each other.

Finally, we apply this to polynomial automorphisms in arbitrary characteristic and obtain a formula for the inverse of such a polynomial automorphism.
\end{abstract}

\maketitle

\section{Introduction}

In analysis there is the well-known chain rule formula for differentiable functions $f:U\subseteq \RR^n\to \RR^m$ and $g:V\subseteq \RR^m\to \RR^l$ stating that the total derivative $D_a(g\circ f)$ of $g\circ f$ at a point $a\in U$ is the composition of the total derivatives $D_{f(a)}(g)$ and $D_a(f)$, i.e.
\[ D_a(g\circ f)=D_{f(a)}(g)\circ D_a(f). \]
In terms of the partial derivatives this is
\[ J_a(g\circ f)=J_{f(a)}(g)\cdot J_a(f), \]
where 
\[ J_a(f)= \begin{pmatrix} \frac{\partial f_1}{\partial x_1}(a) & \frac{\partial f_1}{\partial x_2}(a) &\cdots & \frac{\partial f_1}{\partial x_n}(a)\\
\vdots & \vdots & & \vdots \\
\frac{\partial f_m}{\partial x_1}(a) & \frac{\partial f_m}{\partial x_2}(a) &\cdots & \frac{\partial f_m}{\partial x_n}(a)
\end{pmatrix} \]
is the Jacobian matrix of $f$ at $a\in U$.

In the one-dimensional case, di Bruno \cite{fb:nnfcd} generalized this chain rule to higher derivatives giving the formula
\begin{equation}\label{eq:faa-di-bruno-formula}
 \tfrac{d^j}{dx^j}(g\circ f)(x)= 
 \sum_{\substack{b_1,\ldots, b_j\in \NN\\ \sum ib_i=j}} \tfrac{j!}{b_1!\cdots b_j!} g^{(\sum b_i)}(f(x))
\left( \tfrac{f'(x)}{1!} \right)^{b_1}\left( \tfrac{f''(x)}{2!} \right)^{b_2}\cdots \left( \tfrac{f^{(j)}(x)}{j!} \right)^{b_j}.
\end{equation} 
This formula was generalized also to arbitrary dimensions in \cite[Formula B]{lef:ffhdcf}.
Using  the usual multiindex notation as well as $D^\vect{\alpha}=(\frac{\partial}{\partial x_1})^{\alpha_1}\circ  \cdots
\circ (\frac{\partial}{\partial x_n})^{\alpha_n}$ for $\vect{\alpha}=(\alpha_1,\ldots, \alpha_n) \in \NN^n$ the formula reads
\begin{equation}\label{eq:faa-di-bruno-formula-general}
D^{\vect{\alpha}}(g\circ f)(x) = \hspace*{-2pt}
\sum_{\substack{\vect{\lambda}\in \NN^m\\ 1\leq |\vect{\lambda}|\leq |\vect{\alpha}|}} \hspace*{-2pt}
\tfrac{\vect{\alpha}!}{\vect{\lambda}!} D^\vect{\lambda}(g) (f(x)) \hspace*{-4pt}
\sum_{\vect{\gamma_1}\!+\!\ldots +\vect{\gamma_m}=\vect{\alpha}}
\hspace*{-5mm} P_{\vect{\gamma_1}}(\lambda_1,f_1;x)\cdots P_{\vect{\gamma_m}}(\lambda_m,f_m;x)
\end{equation}
where for $B_{\vect{\gamma}}=\{ \vect{\beta}\in \NN^n \mid \vect{0}< \vect{\beta}\leq 
\vect{\gamma} \}$, $\mu\in \NN$ and $v:U\to \RR$ one defines
\[  P_{\vect{\gamma}}(\mu, v; x):= \sum_{\substack{(\rho_{\vect{\beta}})_{\vect{\beta}\in B_{\vect{\gamma}}}: \sum \rho_{\vect{\beta}}=\mu \\ \sum \rho_{\vect{\beta}}\vect{\beta}=\vect{\gamma}
}} \frac{\mu !}{ \prod_{\vect{\beta}\in B_{\vect{\gamma}}} \rho_{\vect{\beta}}!} \prod_{\vect{\beta}\in B_{\vect{\gamma}}} \left(\frac{D^{\vect{\beta}}(v)(x)}{\vect{\beta}!}\right)^{\rho_{\vect{\beta}}}. \]

These analytic formulas can be translated directly to the algebraic setting over any field of characteristic zero:\\
Let $K$ be a field of characteristic zero, and let $f:U\to \AA_K^m$ and $g:V\to \AA_K^l$ be morphisms of algebraic varieties where $U\subseteq \AA_K^n$ and $V\subseteq \AA_K^m$ are Zariski-open subsets. Then 
$ J_a(g\circ f)=J_{f(a)}(g)\cdot J_a(f)$, and also Formula \eqref{eq:faa-di-bruno-formula-general} holds. Whilest the chain rule for the Jacobian matrices also holds in positive characteristic, Fa\`a di Bruno's formula and Formula \eqref{eq:faa-di-bruno-formula-general} can only  be transfered to positive characterstic $p$ when $p$ is bigger than the highest derivation order, since otherwise one would divide by $p$.

\smallskip

As it will be shown in this article, this problem can be circumvented by using iterative higher derivations, as it is already done for the one-dimensional case in \cite[Prop.~7.2]{ar:icac}.
 Roughly speaking, an iterative higher derivation is a family of maps $\theta^{(n)}$ for all $n\in \NN$ resembling the family $\frac{1}{n!}\partial^n$ for a given derivation $\partial$, but defined also in positive characteristic. Formally replacing all $\frac{1}{n!}\left(\tfrac{\partial}{\partial x_j}\right)^n$ by  symbols $\theta_j^{(n)}$, Formula \eqref{eq:faa-di-bruno-formula-general}
 turns into a formula with integer coefficients, and hence makes sense in any characteristic. Indeed we will show in Section \ref{sec:chain-rule-formula} that this analog holds in positive characteristics (cf.~Thm.~\ref{thm:chain-rule}), generalising the one-dimensional case of the formula in \cite[Prop.~7.2]{ar:icac}.  We will also give a presentation of the formula which is much easier to remember than the one above.


\smallskip

The chain rule formula given in Thm.~\ref{thm:chain-rule} will be used in Section \ref{sec:coordinate-change} to express the higher derivations with respect to one coordinate system $s_1,\ldots, s_n$ in terms of the higher derivations with respect to another coordinate system $t_1,\ldots, t_n$, and leads to a criterion for extending higher derivations to ring extensions.

\smallskip

In Cor.~\ref{cor:extension-in-polynomial-case}, we apply this to the case of an extension of polynomial rings $K[x_1,\ldots,x_n]\supseteq K[f_1,\ldots, f_n]$ whose Jacobian matrix is invertible, to obtain a unique extension of the higher derivations with respect to $f_1,\ldots, f_n$ to the polynomial ring $K[x_1,\ldots,x_n]$. Actually, we get this extendability for any commutative ring $K$ by the fact that invertibility of the  Jacobian matrix implies \'etaleness of the extension $K[x_1,\ldots,x_n]\supseteq K[f_1,\ldots, f_n]$ (cf.~Prop.~\ref{prop:unique-extension-over-ring}). This gives a much shorter proof of \cite[Thm.~2.13]{sk-so:pcawkc}.

Further on in Section \ref{sec:inversion-formula}, we apply our results to the case of polynomial endomorphisms and polynomial automorphisms. In particular, we give a formula for the inverse of an automorphism of a polynomial ring using higher derivatives, analogous to the one given for characteristic zero in \cite[Prop.~3.1.4(ii)]{avde:pajc}.

\section{Basic notation}\label{sec:notation}

All rings are assumed to be commutative with unit and different from $\{0\}$.
The set of natural numbers $\NN$ contains $0$.
Throughout the paper, $K$ will denote a field of arbitrary characteristic.

\subsection{Multi-index notation and evaluation maps}
As we will be dealing with multi-indices, we introduce the usual multi-index notation.

For $\vect{\alpha}=(\alpha_1,\ldots, \alpha_d)\in \NN^d$ and $\vect{\beta}=(\beta_1,\ldots, \beta_d)\in \NN^d$ we write
\begin{itemize}
\item $|\vect{\alpha}|=\alpha_1+\alpha_2+\ldots+\alpha_d$,
\item $\vect{\alpha}!=\alpha_1!\alpha_2!\cdots \alpha_d!$,
\item $\binom{\vect{\alpha}}{\vect{\beta}}=\binom{\alpha_1}{\beta_1}\binom{\alpha_2}{\beta_2}\cdots \binom{\alpha_d}{\beta_d}$, 
\item $\vect{\alpha}\leq \vect{\beta}$ if and only if $\alpha_i\leq \beta_i$ for all $i\in\{1,\ldots, n\}$,
\item $\vect{\alpha}< \vect{\beta}$ if and only if $\vect{\alpha}\leq \vect{\beta}$, but not equal,
\item sums $\vect{\alpha}+ \vect{\beta}$ and differences $\vect{\alpha}- \vect{\beta}$  are componentwise,
\item $\vect{0}=(0,\ldots, 0)$,
\item $\vect{r}^\vect{\alpha}=r_1^{\alpha_1}r_2^{\alpha_2}\cdots r_d^{\alpha_d}$ where
$\vect{r}=(r_1,\ldots, r_d)$ is a tuple of elements in a ring.
\end{itemize}

We will often replace indeterminants by other expressions/values. We will write
\[ H|_{T_1=a_1,\ldots, T_d=a_d} \quad \text{ or }\quad H|_{\vect{T}=\vect{a}} \] 
for short,  for evaluating a function $H$ depending on indeterminants $T_1,\ldots, T_d$ at $(a_1,\ldots, a_d)$, i.e.~replacing the $T_i$ by $a_i$.

\subsection{Higher derivations}

In the introduction we already mentioned that we need a collection of maps in positive characteristic which resembles the family $\left( \frac{1}{n!}\partial^n\right)_{n\in \NN}$. These families and some generalization will be introduce here. See also \cite[Sect.~27]{hm:crt} for univariate higher derivations and \cite[Sect.~3]{am:gticngg} for a more general notion of higher derivations. $d$-variate iterative higher derivations are treated in more detail in \cite{fh:pvtlpd}.

\begin{defn}
A {\markdef $d$-variate higher derivation} on a ring $R$ is a ring homomorphism
$\theta:R\to R[[T_1,\ldots, T_d]]$ such that $\theta(r)|_{\vect{T}=\vect{0}}=r$ for all $r\in R$.
For all $\vect{\alpha}=(\alpha_1,\ldots, \alpha_d)\in \NN^d$, we get induced additive maps $\th{\vect{\alpha}}:R\to R$ given by mapping $r\in R$ to the coefficient of
$\vect{T}^\vect{\alpha}=T_1^{\alpha_1}\cdots T_d^{\alpha_d}$ in $\theta(r)$. Hence for all $r\in R$ one has
\[ \theta(r)=\sum_{\vect{\alpha}\in \NN^d} \th{\vect{\alpha}}(r) \vect{T}^\vect{\alpha}. \]
A $d$-variate higher derivation $\theta$ is called a $d$-variate {\markdef iterative higher derivation} if for all $\vect{\alpha},\vect{\beta}\in \NN^d$ one has
\[ \th{\vect{\alpha}}\circ \th{\vect{\beta}}= \binom{\vect{\alpha}+\vect{\beta}}{\vect{\alpha}}\th{\vect{\alpha}+\vect{\beta}}. \]
\end{defn}


\begin{rem}
As it is easily seen, a $d$-variate higher derivation can equivalently be defined by a collection of additive maps $(\th{\vect{\alpha}}:R\to R)_{\vect{\alpha}\in \NN^d}$ such that
\begin{enumerate}
\item $\th{\vect{0}}=\id_R$,
\item\label{item:leibniz-rule} $\th{\vect{\gamma}}(rs)=\sum_{\vect{\alpha}+\vect{\beta}=\vect{\gamma}}\th{\vect{\alpha}}(r)\th{\vect{\beta}}(s)$ for all $r,s\in R$, $\vect{\gamma}\in \NN^d$,
\end{enumerate}
and $\theta:R\to R[[T_1,\ldots, T_d]], r\mapsto \sum_{\vect{\alpha}\in \NN^d} \th{\vect{\alpha}}(r) \vect{T}^\vect{\alpha}$.

In the univariate (=$1$-variate) case, this is exactly the usual definition of a higher derivation.
\end{rem}

%

\begin{lem}\label{lem:d-variate-higher-derivation}
Let $\theta_1,\ldots, \theta_d$ be $d$ univariate higher derivations on a ring $R$, and let $T_1,\ldots, T_d$ be indeterminates.
Then the map
\[ \theta:R\to R[[T_1,\ldots, T_d]], r\mapsto \sum_{\alpha_1,\ldots,\alpha_d\in \NN}
(\th{\alpha_1}_1\circ \cdots \circ \th{\alpha_d}_d) (r)T_1^{\alpha_1}\cdots T_d^{\alpha_d}
 \] 
 is a $d$-variate higher derivation, and the higher derivations $\theta_1,\ldots, \theta_d$ are determined by $\theta$.
\end{lem}

\begin{proof}
The map $\theta$ is just the composition of ring homomorphisms
\begin{eqnarray*} 
& R\xrightarrow{\theta_d} R[[T_d]] &\xrightarrow{\theta_{d-1}[[T_d]]}R[[T_{d-1},T_d]]
\xrightarrow{\theta_{d-2}[[T_{d-1},T_d]]}  R[[T_{d-2}, T_{d-1}, T_d]] \\
&& \to \ldots \, \to R[[T_2,\ldots, T_d]]\xrightarrow{\theta_{1}[[T_{2},\ldots ,T_d]]} \quad R[[T_1,\ldots, T_d]] 
\end{eqnarray*}
where $\theta_{d-1}[[T_d]]$ denotes the $T_d$-linear extension of the higher derivation
$\theta_{d-1}:R\to R[[T_{d-1}]]$, and $\theta_{d-2}[[T_{d-1},T_d]]$ denotes the $T_{d-1}$-linear and $T_d$-linear extension of the higher derivation
$\theta_{d-2}:R\to R[[T_{d-2}]]$ etc.
Also by definition, $\theta(r)|_{\vect{T}=\vect{0}}=(\th{0}_1\circ \cdots \circ \th{0}_d) (r) = r$ for all $r\in R$. Hence, $\theta$ is a higher derivation. \\
Since the coefficient of $T_i^m$ in $\theta(r)$ is by definition just $\th{m}_i(r)$, the homomorphism $\theta$ determines the higher derivations $\theta_1,\ldots, \theta_d$.\\
Actually, one can also obtain $\theta_i$ by composing $\theta$ with the evaluation homomorphism sending $T_j$ to $0$ for $j\ne i$.
\end{proof}

\begin{exmp}\label{ex:ID-rings}
\begin{enumerate}
\item\label{item:der by t_i} For any field $K$ and $R:=K[t_1,\ldots, t_d]$, the higher derivation
$\theta_{t_i}:R \to R[[T]]$ ($i\in \{1,\ldots d\}$) 
given by $\theta_{t_i}(t_i):=t_i+T$ and $\theta_{t_i}(t_j)=t_j$ if $j\ne i$ is called the {\markdef (iterative) higher derivation with respect to~$t_i$}. The $K$-linear maps $\th{n}_{t_i}$ can explicitly be described as
\[ \th{n}_{t_i}\left( t_1^{\alpha_1}t_2^{\alpha_2}\cdots t_d^{\alpha_d}\right)=\binom{\alpha_i}{n}t_1^{\alpha_1}t_2^{\alpha_2}\cdots t_i^{\alpha_i-n}\cdots t_d^{\alpha_d}. \]
Hence, the first map $\th{1}_{t_i}$ is nothing else than the usual partial derivation $\frac{\partial}{\partial t_i}$, and in characteristic zero, the $\th{n}_{t_i}$ equal $\frac{1}{n!}\left(\frac{\partial}{\partial t_i}\right)^n$.
\item\label{item:der by t} Also for $R:=K[[t_1,\ldots, t_d]]$ the iterative higher derivations
$\theta_{t_i}:R \to R[[T]]$ ($i\in \{1,\ldots d\}$) 
given by $\theta_{t_i}(t_i):=t_i+T$ and $\theta_{t_i}(t_j)=t_j$ and continuous extension are well-defined.
\item\label{item:multi-der by t} By Lemma \ref{lem:d-variate-higher-derivation}, the higher derivations $\theta_i$ give rise to a $d$-variate higher derivation on $R:=K[t_1,\ldots, t_d]$ and $R:=K[[t_1,\ldots, t_d]]$ which will be denoted by
$\theta_{\vect{t}}$ where $\vect{t}$ is the tuple $(t_1,\ldots, t_d)$.
More detailed, $\theta_{\vect{t}}:R\to R[[T_1,\ldots, T_d]]$ is
given by
\[ \theta_{\vect{t}}(t_i)=t_i+T_i \quad \text{ for }i=1,\ldots, d. \]
This higher derivation is even an iterative higher derivation which follows from the following remark.
\end{enumerate}
\end{exmp}

\begin{rem}
Let $\theta_1,\ldots, \theta_d:R\to R[[T]]$ be univariate higher derivations, and $\theta:R\to R[[T_1,\ldots,T_d]]$ the $d$-variate higher derivation composed of them as in Lemma \ref{lem:d-variate-higher-derivation}. Then $\theta$ is an iterative higher derivation if and only if the $\theta_i$ are iterative higher derivations which pairwise commute, i.e.~such that $\th{n}_i\circ \th{m}_j= \th{m}_j\circ \th{n}_i$ for all $n,m\in \NN$, and $i,j\in \{1,\ldots, d\}$. The easy proof can be found in \cite[Satz 4.1.7]{fh:pvtlpd}.

In \cite{fh:pvtlpd}, it is also shown that the condition for being iterative $\th{\vect{\alpha}}\circ \th{\vect{\beta}}= \binom{\vect{\alpha}+\vect{\beta}}{\vect{\alpha}}\th{\vect{\alpha}+\vect{\beta}}$ for all $\vect{\alpha},\vect{\beta}\in \NN^d$ can be expressed equivalently by requiring that the diagram

\centerline{
\xymatrix{ R \ar[r]^(.4){\theta} \ar[d]^{\theta} & R[[\vect{T}]] \ar[r]^{\vect{T}\mapsto \vect{U}} & R[[\vect{U}]] \ar[d]^{\theta[[\vect{U}]]} \\
R[[\vect{T}]]  \ar[rr]^{\vect{T}\mapsto \vect{T}+\vect{U}} && R[[\vect{T},\vect{U}]]
}}

commutes. Here, $\vect{U}=(U_1,\ldots,U_d)$ denotes another $d$-tuple of indeterminates, and as in the proof of Lemma \ref{lem:d-variate-higher-derivation}, $\theta[[\vect{U}]]:R[[\vect{U}]]\to R[[\vect{T},\vect{U}]]$ denotes the $\vect{U}$-linear extension of $\theta$. 
\end{rem}

We will also be concerned with higher derivations on other rings. To obtain those we will make use of the following lemma which is given in \cite[Thm.~27.2]{hm:crt} for the univariate case and in \cite[Prop.~3.7]{am:gticngg} in general. However, the proof for the univariate case easily  generalizes to the multivariate case.
 
\begin{lem}\label{lem:unique-extension}
Let $S/R$ be an \'etale extension, and $\psi:R\to R[[T_1,\ldots,T_d]]$ a $d$-variate higher derivation. Then
\begin{enumerate}
\item $\psi$ uniquely extends to a $d$-variate higher derivation on $S$. In particular, this applies to localizations and to separable field extensions.
\item If $\psi$ is iterative, then its unique extension is iterative, too.
\end{enumerate} 
\end{lem}

\begin{defn}\label{def:theta-t}
Let $R$ be an \'etale extension of the polynomial ring $K[t_1,\ldots, t_d]$, and $\vect{t}=(t_1,\ldots, t_d)$. Then the unique extension of the $d$-variate iterative higher derivation $\theta_{\vect{t}}$ above to a $d$-variate iterative higher derivation on $R$ will also be denoted by
$\theta_{\vect{t}}:R\to R[[T_1,\ldots, T_d]]$. 
\end{defn}

\begin{defn}
Let $R$ be an \'etale extension of the polynomial ring $K[t_1,\ldots, t_d]$, and
$\theta_{\vect{t}}$ the higher derivation with respect to $\vect{t}=(t_1,\ldots, t_d)$ on $R$. Then for any $n$-tuple $\vect{f}=(f_1,\ldots,f_n)$ of elements of $R$ we denote the Jacobian matrix of $\vect{f}$ with respect to $\vect{t}$ by
\[ \Jac{f}{t}:=\begin{pmatrix} \th{1}_{t_1}(f_1) &  \th{1}_{t_2}(f_1) &\cdots &  \th{1}_{t_n}(f_1) \\
\vdots & \vdots & & \vdots \\
 \th{1}_{t_1}(f_n) &  \th{1}_{t_2}(f_n) &\cdots &  \th{1}_{t_n}(f_n)
\end{pmatrix}=  \begin{pmatrix} \frac{\partial f_1}{\partial t_1} & \frac{\partial f_1}{\partial t_2} &\cdots & \frac{\partial f_1}{\partial t_n}\\
\vdots & \vdots & & \vdots \\
\frac{\partial f_n}{\partial t_1} & \frac{\partial f_n}{\partial t_2} &\cdots & \frac{\partial f_n}{\partial t_n}
\end{pmatrix}  . \]
\end{defn}

\section{Chain rule formula}\label{sec:chain-rule-formula}

We are now prepared to state and to proof the chain rule formula.

\begin{thm}[Chain rule]\label{thm:chain-rule}
Let $U\subseteq \AA_K^n$ and $V\subseteq \AA_K^m$ be Zariski-open subsets, and
let $f:U\to \AA_K^m$ and $g:V\to \AA_K^l$ be morphisms of algebraic varieties given by tuples of rational functions $\vect{f}=(f_1(x_1,\ldots, x_n),\ldots, f_m(x_1,\ldots, x_n))$  and $\vect{g}=(g_1(y_1,\ldots, y_m),\ldots, g_l(y_1,\ldots, y_m))$,
 respectively. Further, let $\theta_{\vect{x}}$ be the $n$-variate higher derivation on $\Gamma(U,\O_U)\subseteq K(x_1,\ldots,x_n)$ and $\theta_{\vect{y}}$ 
 the $m$-variate higher derivation on $\Gamma(V,\O_V)\subseteq K(y_1,\ldots,y_m)$ defined as in Def.~\ref{def:theta-t}. 
  Then we have\footnote{Here and in the following, application of a higher derivation to a vector is meant componentwise, e.g.~$\theta_{\vect{x}}(\vect{f})=(\theta_{\vect{x}}(f_1),\ldots, \theta_{\vect{x}}(f_n))$.}
 \begin{eqnarray*}  \theta_{\vect{x}}(g\circ f) &=& (\theta_{\vect{y}}(g)|_{\vect{y}=\vect{f}})|_{\vect{T}= \theta_{\vect{x}}(\vect{f})-\vect{f}} \\
 &=&  \sum_{\vect{\mu}\in \NN^m} \theta_{\vect{y}}^{(\vect{\mu})}(g)|_{\vect{y}=\vect{f}}\cdot  \prod_{j=1}^m \left(
 \sum_{\substack{\vect{\nu}\in \NN^n\\ \vect{\nu}>\vect{0}}}  \theta_{\vect{x}}^{(\vect{\nu})}(f_j) \vect{T}^\vect{\nu} \right)^{\mu_j}
 \end{eqnarray*}
\end{thm}

\begin{proof}
$\theta_{\vect{x}}(g\circ f) $ is the image of $g=(g_1,\ldots, g_l)$ via the ring homomorphism
\[K(y_1,\ldots, y_m)\xrightarrow{\vect{y}\mapsto \vect{f}} K(x_1,\ldots, x_n)\xrightarrow{\theta_\vect{x}}
K(x_1,\ldots, x_n)[[T_1,\ldots, T_n]],\]
and $(\theta_{\vect{y}}(g)|_{\vect{y}=\vect{f}})|_{\vect{T}= \theta_{\vect{x}}(\vect{f})-\vect{f}}$ is its image via the ring homomorphism
\begin{eqnarray*} K(y_1,\ldots, y_m) &\xrightarrow{\theta_{\vect{y}}}& K(y_1,\ldots, y_m)[[T_1,\ldots, T_m]]\\
&\xrightarrow{\vect{y}\mapsto \vect{f}}& K(x_1,\ldots, x_n)[[T_1,\ldots, T_m]] \\ 
&\xrightarrow{\vect{T}\mapsto \theta_\vect{x}(\vect{f})-\vect{f}}&
K(x_1,\ldots, x_n)[[T_1,\ldots, T_n]].
\end{eqnarray*}
Hence, it is sufficient to show that both homomorphisms are the same which amounts to say that the images of each $y_j$ are the same. But $y_j$ maps to $\theta_{\vect{x}}(f_j)$ under the first homomorphism, and $y_j$ maps to
\[ (\theta_{\vect{y}}(y_j)|_{\vect{y}=\vect{f}})|_{\vect{T}= \theta_{\vect{x}}(\vect{f})-\vect{f}}
\!=\!\left( (y_j+T_j)|_{\vect{y}=\vect{f}}\right)|_{\vect{T}= \theta_{\vect{x}}(\vect{f})-\vect{f}}
\!=\! f_j+(\theta_{\vect{x}}(f_j)-f_j)\!=\!\theta_{\vect{x}}(f_j) \]
via the second homomorphism.
Hence, the claim follows.
\end{proof}

\begin{rem}\label{rem:details-for-higher-derivatives}
From the equation above, one can extract the formulas for the $\th{\vect{\lambda}}_{\vect{x}}(g\circ f) $
for any $\vect{0}<\vect{\lambda}\in \NN^n$ by collecting all terms contributing to $\vect{T}^\vect{\lambda}$. We get the same formulas as in \cite{lef:ffhdcf}, since the computation is merely the same. In detail we obtain:

\[ \th{\vect{\lambda}}_{\vect{x}}(g\circ f) = \hspace*{-2mm} \sum_{0<|\vect{\mu}|\leq |\vect{\lambda}|}  \hspace*{-1mm} \theta_{\vect{y}}^{(\vect{\mu})}(g)|_{\vect{y}=\vect{f}}\hspace*{-3pt} \sum_{\vect{\nu}_1+\ldots +\vect{\nu}_m=\vect{\lambda}} \hspace*{-4mm} P_{\vect{\nu}_1}(\lambda_1,f_1)\cdot  P_{\vect{\nu}_2}(\lambda_2,f_2)\cdots  P_{\vect{\nu}_m}(\lambda_m,f_m), \]
where for $B_\vect{\nu}=\{ \vect{\beta}\in \NN^n | \vect{0}< \vect{\beta}\leq \vect{\nu} \}$ and\\
$R_{\vect{\nu},l}=\{  (\rho_{\vect{\beta}})_{\vect{\beta}\in B_\vect{\nu}}\in \NN^{\#B_\vect{\nu}} \mid \sum\limits_{\vect{\beta}\in B_\vect{\nu}}\rho_{\vect{\beta}}=l,\, \sum\limits_{\vect{\beta}\in B_\vect{\nu}}\rho_{\vect{\beta}}\vect{\beta}=\vect{\nu} \}$:
\[ P_{\vect{\nu}}(l,h)=\sum_{(\rho_\vect{\beta})\in R_{\vect{\nu},l}} \frac{l!}{\prod_{\vect{\beta}\in B_\vect{\nu}} \rho_\vect{\beta}!}
\prod_{\vect{\beta}\in B_\vect{\nu}} \left(\th{\vect{\beta}}_{\vect{x}}(h)\right)^{\rho_{\vect{\beta}}}. \]

Be aware that for all $(\rho_\vect{\beta})\in R_{\vect{\nu},l}$ the fraction $ \frac{l!}{\prod_{\vect{\beta}\in B_\vect{\nu}} \rho_\vect{\beta}!}$ is a multinomial coefficient and hence an integer. So its residue modulo the characteristic $p$ is well defined.

In the one-dimensional case ($m=n=1$) one obtains the formula already present in \cite[Prop.~7.2]{ar:icac} which is the direct analog to the Taylor expansion approach to Fa\`a di Bruno's formula in the case of real functions.
\end{rem}

Considering only the first derivatives, one ends up with the well-known formula for the Jacobian matrices:
\[  J(g\circ f)=J(g)|_{\vect{y}=\vect{f(x)}}\cdot J(f). \]

%

\section{Coordinate change}\label{sec:coordinate-change}

In this section we focus on the higher derivations themselves.

Let $F$ be a field extension of $K$ of transcendence degree $n$. Hence every transcendence basis consists of $n$ elements, and we have seen in Example \ref{ex:ID-rings}\ref{item:multi-der by t} and Def.~\ref{def:theta-t} that any separating transcendence basis $t_1,\ldots, t_n$ gives rise to a $n$-variate higher derivation
$\theta_{\vect{t}}$ on $F$.
If $s_1,\ldots, s_n$ is another separating transcendence basis with higher derivation $\theta_{\vect{s}}$ on $F$, the question is how both higher derivations are related to each other, and in particular, how one can get $\theta_{\vect{s}}$ in terms of $\theta_{\vect{t}}$.

\begin{thm}\label{thm:change-of-coordinates}
Let $F$  be a field extension of $K$ of transcendence degree $n$ and let $t_1,\ldots, t_n$ as well as $s_1,\ldots, s_n$ be two separating transcendence bases of $F$.\footnote{We assume here that $F$ has separating transcendence bases. This is guaranteed for example in the case that $K$ is perfect and $F$ is a finitely generated field extension of $K$.} Then the higher derivation with respect to $\vect{s}=(s_1,\ldots, s_n)$ can be expressed in terms of the higher derivation with respect to $\vect{t}=(t_1,\ldots, t_n)$ and the values $\theta_{\vect{s}}^{(\vect{\nu})}(t_j)$. More precisely, for any $g\in F$ one has 
\begin{eqnarray}
 \theta_{\vect{s}}(g) &=&   \theta_{\vect{t}}(g)|_{\vect{T}= \theta_{\vect{s}}(\vect{t})-\vect{t}} \notag \\  
 &=& \sum_{\vect{\mu}\in \NN^n} \theta_{\vect{t}}^{(\vect{\mu})}(g)\cdot  \prod_{j=1}^n \left(
 \sum_{\vect{0}<\vect{\nu}\in \NN^n}  \theta_{\vect{s}}^{(\vect{\nu})}(t_j) \vect{T}^\vect{\nu} \right)^{\mu_j}. \label{eq:formula-for-theta-t-s}
\end{eqnarray}
\end{thm}

\begin{proof}
The proof is merely the same as for Theorem \ref{thm:chain-rule}.
The composition
\[  F\xrightarrow{\theta_{\vect{t}}} F[[\vect{T}]]\xrightarrow{\vect{T}\mapsto \theta_{\vect{s}}(\vect{t})-\vect{t}}  F[[\vect{T}]] \]
is a ring homomorphism whose composition with the evaluation map $\vect{T}\mapsto \vect{0}$ is the identity on $F$, i.e.~it is a $n$-variate higher derivation. By Lemma \ref{lem:unique-extension}, it equals $\theta_{\vect{s}}$ if and only if their restrictions to $K(t_1,\ldots, t_n)$ agree. But for every $i\in \{1,\ldots,n\}$ one has
\[  \theta_{\vect{t}}(t_i)|_{\vect{T}= \theta_{\vect{s}}(\vect{t})-\vect{t}}
= (t_i+T_i)|_{\vect{T}= \theta_{\vect{s}}(\vect{t})-\vect{t}}=t_i+\theta_{\vect{s}}(t_i)-t_i=\theta_{\vect{s}}(t_i). \]
\end{proof}

Exploring Identity \eqref{eq:formula-for-theta-t-s}, we get as the coefficient of $T_j$ the identity
\begin{equation}\label{eq:1st-partials}
 \theta_{\vect{s}}^{(\vect{e_j})}(g)= \th{1}_{s_j}(g)= \sum_{i=1}^n \th{1}_{t_i}(g)  \th{1}_{s_j}(t_i),
\end{equation}
where $\vect{e_j}=(0,\ldots, 1,\ldots, 0)$ is the $j$-th unit vector.
 
In particular, by replacing $g$ by $s_l$ for $l=1,\ldots, n$ and using $\th{1}_{s_j}(s_l)=\delta_{jl}$, we obtain the well-known formula for Jacobian matrices:
\begin{lem}
Let $F$ be a field extension of $K$ of transcendence degree $n$, and let
$s_1,\ldots, s_n$ as well as $t_1,\ldots, t_n$ be two separating transcendence bases.
Then 
\[ \Jac{t}{s}= \begin{pmatrix} \th{1}_{s_1}(t_1) &  \th{1}_{s_2}(t_1) &\cdots &  \th{1}_{s_n}(t_1) \\
\vdots & \ddots & & \vdots \\
 \th{1}_{s_1}(t_n) &  \th{1}_{s_2}(t_n) &\cdots &  \th{1}_{s_n}(t_n)
\end{pmatrix} \]
is invertible with inverse matrix $\Jac{s}{t}$.
\end{lem}


Exploring Identity \eqref{eq:formula-for-theta-t-s} even more, we get

\begin{thm}\label{thm:inversion-formula-higher-derivatives}
Let $F$ be a field extension of $K$ of transcendence degree $n$, and let
$s_1,\ldots, s_n$ as well as $t_1,\ldots, t_n$ be two separating transcendence bases. Then for all $\vect{\mu}\in \NN^n\setminus \{\vect{0}\}$ 
the higher derivative $\th{\vect{\mu}}_{\vect{t}}(s_i)$ can be computed as a polynomial in 
all $\th{\vect{\lambda}}_{\vect{s}}(t_j)$  with $0<|\vect{\lambda}|\leq |\vect{\mu}|$, $1\leq j\leq n$ and in $\det(\Jac{t}{s})^{-1}$ with integer coefficients. Furthermore, the degree in $\det(\Jac{t}{s})^{-1}$ is bounded by $\binom{n+m}{n+1}$.
\end{thm}

\begin{proof}
We do induction on $m:=|\vect{\mu}|$.\\
We have already seen in the previous lemma that $\Jac{s}{t}=\Jac{t}{s}^{-1}$, and since by Cramer's rule the coefficients of the inverse are given by $\det(\Jac{t}{s})^{-1}$ times a polynomial in the coefficients of $\Jac{t}{s}$, the claim follows for $m=|\vect{\mu}|=1$.

For showing the claim for $m>1$, we use Formula \eqref{eq:formula-for-theta-t-s} for $g=s_i$ for $i\in \{1,\ldots, n\}$:
\[ s_i+T_i =  \sum_{\vect{\mu}\in \NN^n} \theta_{\vect{t}}^{(\vect{\mu})}(s_i)\cdot  \prod_{j=1}^n \left(
 \sum_{\substack{\vect{\nu}\in \NN^n\\ \vect{\nu}>\vect{0}}}  \theta_{\vect{s}}^{(\vect{\nu})}(t_j) \vect{T}^\vect{\nu} \right)^{\mu_j}. \]
Restricting to the total degree $m$, i.e.~ to the terms with monomials $\vect{T}^\vect{\lambda}$ where $|\vect{\lambda}|=m$, we get
\[ 0= \sum_{\substack{\vect{\mu}\in \NN^n\\ |\vect{\mu}|=m}} \theta_{\vect{t}}^{(\vect{\mu})}(s_i)\cdot
\prod_{j=1}^n \left( \sum_{l=1}^n \th{1}_{s_l}(t_j)T_l \right)^{\mu_j}
+ \sum_{\substack{\vect{\mu}\in \NN^n\\ 0<|\vect{\mu}|<m}} \theta_{\vect{t}}^{(\vect{\mu})}(s_i)\cdot
\sum_{\substack{\vect{\lambda}\in \NN^n\\ |\vect{\lambda}|=m}} Q_i(\vect{\mu},\vect{\lambda}) \vect{T}^\vect{\lambda}, \]
for appropriate $Q_i(\vect{\mu},\vect{\lambda})$ which are polynomials in $\th{\vect{\nu}}_{\vect{s}}(t_j)$ ($0<|\vect{\nu}|\leq m$, $1\leq j\leq n$).
Let $c_{\vect{\lambda},\vect{\mu}}$ denote the coefficient of $\vect{T}^\vect{\lambda}$ in
$\prod_{j=1}^n \left( \sum_{l=1}^n \th{1}_{s_l}(t_j)T_l \right)^{\mu_j}$, then
the last equality is equivalent to

\begin{equation}\label{eq:recursion-formula}
\left( c_{\vect{\lambda},\vect{\mu}}\right)_{\vect{\lambda},\vect{\mu}\in I_m} \cdot
\left(   \theta_{\vect{t}}^{(\vect{\mu})}(s_i) \right)_{\vect{\mu}\in I_m} = - \left(  \sum_{\substack{\vect{\mu'}\in \NN^n\\ 0<|\vect{\mu'}|<m}} \theta_{\vect{t}}^{(\vect{\mu'})}(s_i)\cdot
Q_i(\vect{\mu'},\vect{\lambda}) \right)_{\vect{\lambda}\in I_m},
\end{equation}
 where $I_m=\{ \vect{\nu}\in \NN^n \mid  |\vect{\nu}|=m \}$.
 Since by induction all $\theta_{\vect{t}}^{(\vect{\mu'})}(s_i)$ are polynomials in $\th{\vect{\lambda}}_{\vect{s}}(t_j)$ ($0<|\vect{\lambda}|\leq |\vect{\mu'}|$) and in $\det(\Jac{t}{s})^{-1}$, it suffices to show that the matrix $\left( c_{\vect{\lambda},\vect{\mu}}\right)_{\vect{\lambda},\vect{\mu}\in I_m}$ is invertible and that its determinant is a power of $
 \det(\Jac{t}{s})$.
 
Now, let $V$ be the $K$-vector space generated by $T_1,\ldots, T_n$ and consider the endomorphism $\varphi:V\to V, T_j\mapsto \sum_{l=1}^n \th{1}_{s_l}(t_j)T_l$. Then the transpose of $\Jac{t}{s}$ is just the matrix  representing $\varphi$ with respect to the basis $(T_1,\ldots, T_n)$.
By construction  $\left( c_{\vect{\lambda},\vect{\mu}}\right)_{\vect{\lambda},\vect{\mu}\in I_m}$ is nothing else then the matrix representing the $m$-th symmetric power of $\varphi$ with respect to the basis $\{ \vect{T}^{\vect{\mu}}\mid \mu\in I_m\}$. Hence, by Lemma \ref{lem:det-of-symmetric-power} below, we have
\[ \det\left( (c_{\vect{\lambda},\vect{\mu}})_{\vect{\lambda},\vect{\mu}}\right) =\det \left(\Jac{t}{s}\right)^{B(n,m)}\]
with $B(n,m)=\binom{n+m-1}{n}$ as in the lemma.

For getting the bound on the degree with respect to $\det(\Jac{t}{s})^{-1}$, we use Equation \eqref{eq:recursion-formula} to see that for $\vect{\mu}$ with $|\vect{\mu}|=m$ the degree of $\theta_{\vect{t}}^{(\vect{\mu})}(s_i)$ with respect to $\det(\Jac{t}{s})^{-1}$ is at most by $B(n,m)$ larger than the maximum of the degrees of all $\theta_{\vect{t}}^{(\vect{\mu'})}(s_i)$  with $|\vect{\mu'}|<m$. Hence, the degree is at most
$ \sum_{k=1}^{m} \binom{n+k-1}{n}= \binom{n+m}{n+1}.$
\end{proof}

\begin{lem}\label{lem:det-of-symmetric-power}
Let $V$ be a vector space over $K$ of finite dimension $n$, and $\varphi:V\to V$ an endomorphism. Let $S_m(\varphi):S_m(V)\to S_m(V)$ be the induced endomorphism on the $m$-th symmetric power of $V$. Then the determinant of $S_m(\varphi)$ is given by 
\[ \det\left( S_m(\varphi) \right) = \det(\varphi)^{B(n,m)} \]
where $B(n,m)=\binom{n+m-1}{n}$.
\end{lem}

Although this seems to be well known, we  couldn't find a proof of this fact in the literature. Hence, we give one here.

\begin{proof}
Since determinants are invariant under extension of scalars, we can assume that $K$ is algebraically closed, and we can choose a Jordan basis $b_1,\ldots, b_n$ for $\varphi$ with eigenvalues $\lambda_1,\ldots, \lambda_n$. Hence, $\det(\varphi)=\lambda_1\lambda_2\cdots \lambda_n$.\\
A basis of the $m$-th symmetric power is given by $\{ b_1^{\mu_1}\cdots b_n^{\mu_n} \mid \vect{\mu}\in I_m\}$ where $I_m=\{ \vect{\mu}\in \NN^n \mid \sum_{i=1}^n \mu_i=m\}$.
If we order the basis lexicographically the matrix representing $S_m(\varphi)$ is again triangular with diagonal entries
$\{ \lambda_1^{\mu_1}\cdots \lambda_n^{\mu_n} \mid  \vect{\mu}\in I_m\}$.
Hence,
\begin{equation}\label{eq:det-of-symmetric-power}
\det\left( S_m(\varphi)\right) =\prod_{\vect{\mu}\in I_m}  \lambda_1^{\mu_1}\cdots \lambda_n^{\mu_n}. 
\end{equation}
As the last expression is symmetric in $\lambda_1,\ldots, \lambda_n$, it is some power of $\det(\varphi)=\lambda_1\lambda_2\cdots \lambda_n$. 
As the total degree of $\prod_{\vect{\mu}\in I_m}  \lambda_1^{\mu_1}\cdots \lambda_n^{\mu_n}$ is $m\cdot \#I_m$ and the total degree of  $\det(\varphi)=\lambda_1\lambda_2\cdots \lambda_n$ is $n$, the desired power is $\frac{m}{n}\cdot \#I_m$.
Finally $\#I_m=\binom{n+m-1}{m}$, and hence $\det\left( S_m(\varphi) \right) = \det(\varphi)^{B(n,m)}$ for $B(n,m)=\frac{m}{n}\binom{n+m-1}{m}=\binom{n+m-1}{m-1}=\binom{n+m-1}{n}$.
\end{proof}

We have seen in Lemma \ref{lem:unique-extension} that we can extend the higher derivation $\theta_{\vect{t}}$ on $K[t_1,\ldots,t_n]$ to a finite extension $R$ if the extension is \'etale. In the following we have another criterion when a tuple $\vect{t}=(t_1,\ldots,t_n)$ of elements in $R$ define a higher derivation on $R$.

\begin{thm}\label{thm:extension-in-ring-case}
Let $F/K(s_1,\ldots, s_n)$ be a separable algebraic extension and $R\subseteq F$ a $K$-subalgebra such that $\theta_{\vect{s}}:F\to F[[\vect{T}]]$ restricts to a higher derivation on $R$.
Assume that $t_1,\ldots, t_n\in R$ are such that the determinant of $\Jac{t}{s}$ is invertible in $R$. Then
$t_1,\ldots, t_n$ are algebraically independent over $K$ and the higher derivation with respect to $\vect{t}$ on $K[t_1,\ldots, t_n]$ can uniquely be extended to an iterative higher derivation on $R$.
\end{thm}

\begin{proof}
As $F$ is a separable algebraic extension of $K(s_1,\ldots,s_n)$, the differentials $ds_1,\ldots, d s_n$ form a $F$-basis of $\Omega_{F/K}$. As 
$\Jac{t}{s}$ is invertible in $F$, also $dt_1,\ldots, dt_n$ form a $F$-basis of $\Omega_{F/K}$ which means that $F$ is separable algebraic over $K(t_1,\ldots, t_n)$. Hence by Lemma \ref{lem:unique-extension}, the higher derivation with respect to $\vect{t}=(t_1,\ldots, t_n)$ on $K[t_1,\ldots, t_n]$ can uniquely be extended to an iterative higher derivation on $F$, and it remains to prove that $R$ is stable under those.\\
By Theorem \ref{thm:inversion-formula-higher-derivatives}, all $\th{\vect{\mu}}_{\vect{t}}(s_i)$ lie in $R$, as they can be given as polynomials in $\th{\vect{\lambda}}_{\vect{s}}(t_j)\in R$ and in $\det(\Jac{t}{s})^{-1}\in R$.
Finally, Theorem \ref{thm:change-of-coordinates} with roles of $\vect{t}$ and $\vect{s}$ reversed, shows that indeed $\th{\vect{\mu}}_{\vect{t}}(g)$ is an element of $R$ for every $g\in R$.
\end{proof}

As a corollary, we get the following important case of polynomials which is already present in \cite[Thm.~2.13]{sk-so:pcawkc}.

\begin{cor}\label{cor:extension-in-polynomial-case}
Let $x_1,\ldots, x_n$ be indeterminates, and $f_1,\ldots, f_n\in K[x_1,\ldots, x_n]$ such that 
$\det(J(\vect{f},\vect{x}))\in K[x_1,\ldots, x_n]^\times=K^\times$. 
Then the higher derivation $\theta_\vect{f}$ on $K[f_1,\ldots, f_n]$ can uniquely be extended to $K[x_1,\ldots, x_n]$.
\end{cor}

\begin{rem}
Theorem 2.13 in \cite{sk-so:pcawkc} states the unique existence of the extension of  $\theta_\vect{f}$ to $K[x_1,\ldots, x_n]$ for any commutative ring $K$ if additionally the following properties are fulfilled.
\begin{enumerate}
\item If $K$ has characteristic $p^e$ ($p$ prime) then for all $L\geq 0$, $i=1,\ldots, n$ and $N=e+\ord_p(L!)$ one has \[\th{L}_{f_i}\left( \sum_{\vect{0}\leq \vect{\alpha}\leq \vect{p^N\!\!-\!\!1}} \hspace*{-8pt} a_\vect{\alpha}(x_1^{p^N},\ldots, x_n^{p^N}) \vect{f}^\vect{\alpha}\right) = \hspace*{-4pt} \sum_{\vect{0}\leq \vect{\alpha}\leq \vect{p^N\!\!-\!\!1}}\hspace*{-8pt} a_\vect{\alpha}(x_1^{p^N},\ldots, x_n^{p^N}) \binom{\alpha_i}{L} \vect{f}^{\vect{\alpha}-L\vect{e_i}}, \]
\item if $L!$ is not a zero-divisor in $K$, then $L!\cdot \th{L}_{f_i}=\left(\frac{\partial}{\partial f_i}\right)^L$, as well as
\item if $\phi:K\to K'$ is a homomorphism of commutative rings, then the following diagram commutes.

\centerline{
\xymatrix{ K[x_1,\ldots, x_n] \ar[r]^(0.45){\theta_\vect{f}} \ar[d]
& K[x_1,\ldots, x_n][[\vect{T}]] \ar[d] \\
K'[x_1,\ldots, x_n] \ar[r]^(0.45){\theta_{\phi(\vect{f})}} & K'[x_1,\ldots, x_n][[\vect{T}]]
}}
\noindent  where the vertical maps are the homomorphisms induced by $\phi$.
\end{enumerate}

Our approach leads to a much shorter proof of their theorem, and even shows stronger statements which are given in the following proposition.
\end{rem}

\begin{prop}\label{prop:unique-extension-over-ring}
For any commutative ring $A$ and $f_1,\ldots,f_n\in A[x_1,\ldots, x_n]$ such that $\det(J(\vect{f},\vect{x}))\in A[x_1,\ldots, x_n]^\times$ the iterative higher derivation $\theta_\vect{f}:A[f_1,\ldots,f_n]\to A[f_1,\ldots,f_n][[T]]$ uniquely extends to a higher derivation on $A[x_1,\ldots, x_n]$ (also called $\theta_{\vect{f}}$). This extension fulfills the following properties:
\begin{enumerate}
\item \label{item:prime-power} If $A$ has characteristic $p^e$ ($p$ prime) then for all $L\geq 0$, $i=1,\ldots, n$, and $N\in \NN$ such that $p^{N-e+1}>L$, one has \[\th{L}_{f_i}\left( \sum_{\vect{0}\leq \vect{\alpha}\leq \vect{p^N\!\!-\!\!1}} \hspace*{-8pt} a_\vect{\alpha}(x_1^{p^N},\ldots, x_n^{p^N}) \vect{f}^\vect{\alpha}\right) = \hspace*{-4pt} \sum_{\vect{0}\leq \vect{\alpha}\leq \vect{p^N\!\!-\!\!1}} \hspace*{-8pt} a_\vect{\alpha}(x_1^{p^N},\ldots, x_n^{p^N})\binom{\alpha_i}{L} \vect{f}^{\vect{\alpha}-L\vect{e_i}}. \]
In particular, this holds for $N=e+\ord_p(L!)$.
\item \label{item:iterated} $\theta_\vect{f}$ is an iterative higher derivation on $A[x_1,\ldots,x_n]$, in particular for all $L>0$,
\[L!\cdot \th{L}_{f_i}=\left(\tfrac{\partial}{\partial f_i}\right)^L.\]
\item \label{item:functorial}  If $\phi:A\to A'$ is a homomorphism of commutative rings, then the following diagram commutes.

\centerline{
\xymatrix{ A[x_1,\ldots, x_n] \ar[r]^(0.45){\theta_\vect{f}} \ar[d]
& A[x_1,\ldots, x_n][[\vect{T}]] \ar[d] \\
A'[x_1,\ldots, x_n] \ar[r]^(0.45){\theta_{\phi(\vect{f})}} & A'[x_1,\ldots, x_n][[\vect{T}]]
}} 
\noindent where the vertical maps are the homomorphisms induced by $\phi$.
\end{enumerate}
\end{prop}

\begin{proof}
Since the Jacobian matrix $\Jac{f}{x}$ is invertible in $S=A[x_1,\ldots, x_n]$, the extension $R=A[f_1,\ldots, f_n]\hookrightarrow S=A[x_1,\ldots, x_n]=R[x_1,\ldots, x_n]/(f_1,\ldots, f_n)$ is standard smooth with $\Omega_{S/R}=0$ (see \cite[Tag 00T7]{stacks-project}), and hence \'etale.
Therefore by Lemma \ref{lem:unique-extension}, the higher derivation $\theta_{\vect{f}}$ on $R=A[f_1,\ldots, f_n]$ uniquely extends to $S=A[x_1,\ldots, x_n]$ and the extension is again iterative. It remains to verify properties \ref{item:prime-power} and \ref{item:functorial} given above.

For showing Part \ref{item:prime-power}, we recognize that in characteristic $p^e$ one has
$\theta_{\vect{f}}(g^{p^N})=\theta_{\vect{f}}(g)^{p^N}\in (S[[\vect{T}]])^{p^N}\subseteq S[[T_1^{p^{N-e+1}},\ldots, T_n^{p^{N-e+1}}]]$, and hence $\th{j}_{f_i}(x^{p^N})=0$ for all $0<j<p^{N-e+1}$. Applying the Leibniz rule and the definition of $\theta_{\vect{f}}$ on $R=A[f_1,\ldots,f_n]$, we obtain the claimed equation.
As $p^{\ord_p(L!)+1}>L$ for every $L>0$, this holds for $N=e+\ord_p(L!)$.


Part \ref{item:functorial} is clear by uniqueness of the extended $\theta_{\vect{f}}$. 
\end{proof}

\section{Inversion formula for polynomial automorphisms}\label{sec:inversion-formula}

In this section, we will use the previous to compute the inverse map of a polynomial automorphism.

So from now on let $A[x_1,\ldots, x_n]$ be a polynomial ring in $n$ variables over the commutative ring $A$ and $F:A[x_1,\ldots, x_n]\to A[x_1,\ldots, x_n]$ an endomorphism given by the polynomials $f_j(x_1,\ldots, x_n):=F(x_j)\in A[x_1,\ldots, x_n]$.
Further, we assume that the Jacobian matrix
\[ \Jac{f}{x}= \begin{pmatrix} \frac{\partial f_1}{\partial x_1} & \frac{\partial f_1}{\partial x_2} &\cdots & \frac{\partial f_1}{\partial x_n}\\
\vdots & \ddots & & \vdots \\
\frac{\partial f_n}{\partial x_1} & \frac{\partial f_n}{\partial x_2} &\cdots & \frac{\partial f_n}{\partial x_n}
\end{pmatrix} \]
is invertible over $A[x_1,\ldots, x_n]$, i.e.~that $\det(  \Jac{f}{x})\in A[x_1,\ldots, x_n]^\times$.

By Corollary \ref{cor:extension-in-polynomial-case} resp.~Proposition \ref{prop:unique-extension-over-ring}, in this case the higher derivation $\theta_{\vect{f}}$ which a priori is only defined on $A[f_1,\ldots, f_n]$ can be extended  uniquely to a higher derivation on $A[x_1,\ldots, x_n]$.

After the affine transformation $f_i\mapsto f_i-f_i(0,\ldots,0)$ we can furthermore assume that $f_i\equiv 0$ modulo the ideal ${(x_1,\ldots, x_n)}$. In this  case $F$ can be extended continuously to an endomorphism $F:A[[x_1,\ldots, x_n]]\to A[[x_1,\ldots, x_n]]$. Moreover, by the formal inverse function theorem (see \cite[Thm.~1.1.2]{avde:pajc}), this endomorphism is even an automorphism as $\Jac{f}{x}$ is invertible.

\begin{thm}\label{thm:formula-for-formal-inverse}
The inverse $G$ of the automorphism $F:A[[x_1,\ldots, x_n]]\to A[[x_1,\ldots, x_n]]$ is given by
\[ G(h)= \left(\theta_\vect{f}(h)|_{\vect{x}=\vect{0}}\right)_{\vect{T}=\vect{x}}
=\sum_{\vect{\mu}\in\NN^n} \th{\vect{\mu}}_{\vect{f}}(h)|_{\vect{x}=\vect{0}} \vect{x}^\vect{\mu}
 \quad \text{for all }h\in A[[x_1,\ldots, x_n]], \]
 where $\theta_\vect{f}:A[[\vect{x}]]\to A[[\vect{x}]][[\vect{T}]]$ is the continuous extension of $\theta_\vect{f}$ above.
\end{thm}

\begin{proof}
As the evaluation homomorphisms and $\theta_{\vect{f}}$ are ring homomorphisms, $G$ is a ring homomorphism.
Hence, we only have to show that the formula above gives a formal inverse. By definition,
\[ G(F(x_i))=G(f_i)= \left(\theta_\vect{f}(f_i)|_{\vect{x}=\vect{0}}\right)|_{\vect{T}=\vect{x}}
=  \left((f_i+T_i)|_{\vect{x}=\vect{0}}\right)|_{\vect{T}=\vect{x}} =T_i|_{\vect{T}=\vect{x}}=x_i\]
where we used that $f_i(0,\ldots, 0)=0$. Therefore, $G\circ F=\id_{A[[x_1,\ldots, x_n]]}$, showing the claim.
\end{proof}

\begin{rem} \ 
\begin{enumerate}
\item In \cite{pn-ms:appsr}, Nousiainen and Sweedler have given a formula for the inverse which reads in our notation as
\[ G(h)= \sum_{\vect{\alpha}\in\NN^n} \th{\vect{\alpha}}_{\vect{f}}(h) (\vect{x}-\vect{f})^\vect{\alpha}. \]
The advantage of our formula is that it gives the inverse directly as power series which enables us to state and prove Thm.~\ref{thm:formula-for-polynomial-inverse} below.
\item When $A$ is a $\QQ$-algebra, our formula is exactly the same as the formula given in \cite[Thm.~3.1.1]{avde:pajc}.
However, the invertibility criterion given in  \cite[Prop.~3.1.4(i)]{avde:pajc} does not hold in positive characteristic. E.g.~in positive characteristic $p$, there are elements $h$ for which $\th{p^k}(h)$ might be non-zero, although $\th{n}(h)=0$ for some $n<p^k$. However, the analog of \cite[Prop.~3.1.4(ii)]{avde:pajc} holds which is stated in the next theorem.
\end{enumerate}
\end{rem}

\begin{thm}\label{thm:formula-for-polynomial-inverse}
Let $F:A[x_1,\ldots, x_n]\to A[x_1,\ldots, x_n]$ be as above. Put $d:=\deg(F):=\max\{ \deg(f_i) \mid i=1,\ldots,n \}$ and $N:=d^{n-1}$. If $F$ is an automorphism, then its inverse is given by
\[ G:A[x_1,\ldots, x_n]\to A[x_1,\ldots, x_n], h(x_1,\ldots, x_n) \mapsto \sum_{\substack{\vect{\lambda}\in \NN^n\\ |\vect{\lambda}|\leq N}}
\th{\vect{\lambda}}_{\vect{f}}(h)|_{\vect{x}=\vect{0}} \cdot \vect{x}^\vect{\lambda}. \]
\end{thm}
 
\begin{proof}
If $F$ is an automorphism, its inverse has to be the restriction of the formal inverse $G$ of Theorem \ref{thm:formula-for-formal-inverse}. Furthermore, by \cite[Cor.~1.4]{hb-ec-dw:jcrdfei}, the degree of the inverse is at most $N$. 
\end{proof} 

Using Thm.~\ref{thm:inversion-formula-higher-derivatives} and the previous formula for the inverse of a polynomial automorphism, one obtains an algorithm for computing the inverse.\footnote{Although in Thm.~\ref{thm:inversion-formula-higher-derivatives} we have a base field, the formulas given there hold over any ring $A$.}
Any automorphism  $F:A[x_1,\ldots, x_n]\to A[x_1,\ldots, x_n]$ can easily be decomposed  into an affine transformation and a polynomial automorphism sending $(0,\ldots,0)$ to $(0,\ldots,0)$, and whose Jacobian matrix is congruent to the identity matrix modulo the ideal $(x_1,\ldots, x_n)$. We therefore use as input of the algorithm the transformed automorphism.

%
%
%
\bigskip

\textbf{Algorithm:}\\
INPUT: An automorphism $F:A[x_1,\ldots, x_n]\to A[x_1,\ldots, x_n], x_i\mapsto f_i$ such that $f_i(0,\ldots, 0)=0$ for all $i$ and that $\Jac{f}{x}|_{\vect{x}=\vect{0}}=\uob_n$.\\
OUTPUT: A tuple $g_1,\ldots, g_n$ inducing an automorphism $G:A[x_1,\ldots, x_n]\to A[x_1,\ldots, x_n], x_i\mapsto g_i$ which is the inverse of $F$.

Algorithm:
\renewcommand\theenumi{\arabic{enumi}.}
\renewcommand\labelenumi{\theenumi}
\begin{enumerate}
\item Set $d:=\deg(F)$ and $N:=d^{n-1}$.
\item Set $\alpha_{i,\vect{e_j}}:=\partdef{1, & i=j\\ 0, & i\ne j}$ for $i,j=1,\ldots, n$ and $\vect{e_j}$ being the $j$-th standard vector.
\item for $l:=2$ to $N$ do
\begin{enumerate}
\item[ ] Compute for all $\vect{\lambda}\in \NN^n$ with $|\vect{\lambda}|=l$ and $i=1,\ldots, n$:
\[ \alpha_{i,\vect{\lambda}}:=-\sum_{0<|\vect{\nu}|<l} \alpha_{i,\vect{\nu}} \cdot \left( \text{coeff. of }\vect{x}^\vect{\lambda}\text{ in } \vect{f}^\vect{\nu} \right). \]
\end{enumerate}
\item Define for $i=1,\ldots, n$:
\[ g_i:= \sum_{\substack{\vect{\lambda}\in \NN^n\\ 0<|\vect{\lambda}|\leq N}}
\alpha_{i,\vect{\lambda}} \cdot \vect{x}^\vect{\lambda}. \]
\end{enumerate}
\renewcommand\theenumi{(\roman{enumi})}
\renewcommand\labelenumi{\theenumi}

\begin{proof}
We prove that the algorithm really computes the inverse given in Theorem \ref{thm:formula-for-polynomial-inverse},  hence that the $\alpha_{i,\vect{\lambda}}$ are equal to $\th{\vect{\lambda}}_{\vect{f}}(x_i)|_{\vect{x}=\vect{0}}$. This is done by induction on $m=|\vect{\lambda}|$. Also be aware that $\th{\vect{0}}_{\vect{f}}(x_i)|_{\vect{x}=\vect{0}}=x_i|_{\vect{x}=\vect{0}}=0$, so that indeed  the sum for $G(x_i)$ in Thm.~\ref{thm:formula-for-polynomial-inverse} starts at $|\vect{\lambda}|=1$.

For $m=1$, we have $\vect{\lambda}=\vect{e_j}$ for some $j\in \{1,\ldots, n\}$ and
\[ \th{\vect{e_j}}_{\vect{f}}(x_i)|_{\vect{x}=\vect{0}}=\frac{\partial x_i}{\partial f_j}|_{\vect{x}=\vect{0}}=\left. \partdef{1, & i=j\\ 0, & i\ne j} \right\}=
\alpha_{i,\vect{e_j}}, \]
since by assumption $\Jac{x}{f}|_{\vect{x}=\vect{0}}=\left(\Jac{f}{x}|_{\vect{x}=\vect{0}}\right)^{-1}=\uob_n$.

Now let $m>1$. By Equation~\eqref{eq:recursion-formula} we have
\[ 
\left(   \theta_{\vect{f}}^{(\vect{\mu})}(x_i) \right)_{\vect{\mu}\in I_m} = - \left( \left( c_{\vect{\lambda},\vect{\mu}}\right)_{\vect{\lambda},\vect{\mu}\in I_m}\right)^{-1} \cdot  \left(  \sum_{\substack{\vect{\nu}\in \NN^n\\ 0<|\vect{\nu}|<m}} \theta_{\vect{f}}^{(\vect{\nu})}(x_i)\cdot
Q_i(\vect{\nu},\vect{\lambda}) \right)_{\vect{\lambda}\in I_m}, \]
where 
$I_m=\{ \vect{\mu'}\in \NN^n \mid  |\vect{\mu'}|=m \}$, $Q_i(\vect{\nu},\vect{\lambda})$ is the coefficient of $\vect{T}^\vect{\lambda}$ in $\prod_{j=1}^n \left( \theta_{\vect{x}}(f_j)-f_j\right)^{\nu_j}$, and $\left( c_{\vect{\lambda},\vect{\mu}}\right)_{\vect{\lambda},\vect{\mu}\in I_m}$ is the $l$-th symmetric power of the transpose of $\Jac{f}{x}$. By reducing modulo $(x_1,\ldots,x_n)$ and recognizing that modulo $(x_1,\ldots,x_n)$ the matrix $\left( c_{\vect{\lambda},\vect{\mu}}\right)_{\vect{\lambda},\vect{\mu}\in I_m}$  is the identity matrix,
 we get that
\[  \theta_{\vect{f}}^{(\vect{\lambda})}(x_i)|_{\vect{x}=\vect{0}} = - \sum_{\substack{\vect{\nu}\in \NN^n\\ 0<|\vect{\nu}|<m}} \alpha_{i,\vect{\nu}}\cdot Q_i(\vect{\nu},\vect{\lambda})|_{\vect{x}=\vect{0}}. \]
As we have $\left( \theta_{\vect{x}}(f_j)-f_j\right)|_{\vect{x}=\vect{0}}
=\left( f_j(\vect{x}+\vect{T})-f_j(\vect{x})\right)|_{\vect{x}=\vect{0}}=f_j(\vect{T})$, the term
$Q_i(\vect{\nu},\vect{\lambda})|_{\vect{x}=\vect{0}}$ equals the coefficient of $\vect{T}^\vect{\lambda}$ in
$\prod_{j=1}^n f_j(\vect{T})^{\nu_j}$, hence the coefficient of $\vect{x}^\vect{\lambda}$ in $\vect{f}^\vect{\nu}$.

Therefore, $\alpha_{i,\vect{\lambda}}=\th{\vect{\lambda}}_{\vect{f}}(x_i)|_{\vect{x}=\vect{0}}$ for $|\vect{\lambda}|=m$. 
\end{proof}

\begin{rem}
If one naively computes the inverse $G$ by the Ansatz \[ g_i:= \sum_{\substack{\vect{\lambda}\in \NN^n\\ 0<|\vect{\lambda}|\leq N}}
\alpha_{i,\vect{\lambda}} \cdot \vect{x}^\vect{\lambda}\]
and elaborating the conditions on the $\alpha_{i,\vect{\lambda}}$ using
$g_i(f_1,\ldots, f_n)=x_i$, one will get exactly the equations in the algorithm:
\[ \alpha_{i,\vect{\lambda}}=-\sum_{0<|\vect{\nu}|<l} \alpha_{i,\vect{\nu}} \cdot \left( \text{coeff. of }\vect{x}^\vect{\lambda}\text{ in } \vect{f}^\vect{\nu} \right). \]

Hence, the formula for the inverse in Theorem \ref{thm:formula-for-polynomial-inverse} is not good for algorithmic purposes, but only for theoretical purposes. In practice, the computation via Gr\"obner bases given in \cite[Thm.~3.2.1]{avde:pajc} is much faster.
\end{rem}

\bibliographystyle{plain}
\def\cprime{$'$}

\vspace*{.5cm}

\end{document}